\documentclass{amsart}
\usepackage{graphicx}
\usepackage{subfigure}
\usepackage{caption}
\usepackage{amssymb}
\usepackage[matrix, arrow, curve]{xy}
\usepackage{enumerate}
\newtheorem{theorem}{Theorem}[section]
\newtheorem{lemma}[theorem]{Lemma}

\theoremstyle{definition}

\usepackage{pdfsync}

\usepackage{tikz}
\usetikzlibrary{shapes}
\usetikzlibrary{fit}
\usetikzlibrary{decorations.pathmorphing}
\usetikzlibrary{calc,arrows}
\theoremstyle{remark}

\numberwithin{equation}{section}

 \DeclareMathSymbol{\vartriangle}{\mathord}{AMSa}{"4D}
\DeclareMathSymbol{\triangledown}{\mathord}{AMSa}{"4F}

\begin{document}
\title{Polynomials associated with  finite Markov chains.}
%    Information for first author
\author{Philippe Biane}
%    Address of record for the research reported here
\address{IGM, Universit\'e Paris Est, Marne-la-Vall\'ee, FRANCE
}
\email{biane@univ-mlv.fr}

\begin{abstract}
Given a finite Markov chain, 
we investigate the first minors of the transition matrix of a lifting of this Markov chain to covering trees. In a simple case we exhibit a nice factorisation of these minors, and we conjecture that it holds more generally.
\end{abstract}
\maketitle

\section{Introduction}
\label{sec:1}
The famous matrix-tree theorem of Kirchhoff gives a combinatorial formula for the invariant measure of a finite Markov chain in terms of covering trees of the state space of the chain. One can provide a probabilistic interpretation of Kirchhoff's formula by lifting the Markov chain to the set  covering trees of its state space, see e.g. \cite{AT} or \cite{LP}, \S 4.4. This yields a new Markov chain, whose transition matrix can be constructed from the transition matrix of the original Markov chain. In this paper, we investigate the first minors of this new matrix, which are polynomials in the entries of the original  transition matrix. We will see that in a simple case, that of a Markov chain evolving on a ring, these polynomials exhibit a remarkable factorisation. We expect that such factorisations hold in a much more general context. This paper is organized as follows: we start in section 2 by recalling some general facts about finite Markov chains and their invariant measure. In section 3 we describe how to lift the Markov chain to its set of covering trees. In section 4 we introduce a polynomial associated to the Markov chain, and show that in the case of a Markov chain with three states it has a nice factorisation.
We generalize this observation to the case of Markov chains on a ring in section 5, which contains the main result of the paper.

I would like to thank Jim Pitman for pointing out reference \cite{LP} to me.

\section{Finite Markov chains and invariant measures}

We start by recalling some well known facts about finite Markov chains.
\subsection{Transition matrix}
We consider a continuous time Markov chain $M$ on a finite set $X$.
Let $Q=(q_{ij})_{i,j\in X}$ be its matrix of transition rates:
$q_{ij}\geq 0$ if $i\ne j\in X$ and $\sum_jq_{ij}=0$ for all $i$.

\subsection{Invariant measure}
An invariant measure for $M$ (more exactly, for $Q$) is a nonzero vector  $\mu(i),i\in X$, with nonnegative entries such that
$\sum_i\mu(i)q_{ij}=0$ for all $j\in X$. An invariant measure always exists, it is unique up to a multiplicative constant if the chain is irreducible.

\subsection{Projection of a Markov chain}
Let $N$ be a Markov chain on a finite state space $Y$, with transition matrix $R=(r_{kl})_{k,l\in Y}$, and $p:Y\to M$ be a map such that, for all $i,j\in X$ and all $k\in Y$ such that $p(k)=i$, one has

\begin{equation}
q_{ij}=\sum_{l\in p^{-1}(j)}r_{kl}
\end{equation}
then $p(N)$ is a Markov chain on $X$ with transition rates $q_{ij}$. Furthermore, if $\nu$ is an invariant measure for $R$, then $\mu$ defined as

\begin{equation}\label{invmeas}
\mu(i)=\sum_{k\in p^{-1}(i)}\nu(k)
\end{equation}
is an invariant measure for $Q$.

\subsection{Oriented graph and covering trees.}
To the matrix $Q$ is associated a graph $(X,E)$ with $X$ as vertex set, and $E$ as edge set, such that there is  an edge from $i$ to $j$ if and only if $q_{ij}>0$. This graph is oriented, has no multiple edges, and no loops (edges which begin and end at the same  vertex). Let $i\in X$, a  {\sl covering tree of $(X,E)$, rooted at $i$} is  an oriented subgraph of $(X,E)$ which is a tree and such that,  for every  $j\in X$, there is a unique path from $j$ to $i$ in the graph  (paths are oriented).
The Markov chain is irreducible if and only if for all $i,j\in X$ there exists a path from $i$ to $j$ in the graph $(X,E)$.
If this is the case then for every vertex $i\in X$ there exists a covering tree rooted at $i$.

Figure 1 shows an oriented graph, together with a covering tree rooted at the shaded vertex (beware that a Markov chain corresponding to this graph is not irreducible).
\medskip

\begin{figure}
\centering
\begin{tikzpicture}[->,>=stealth',shorten >=1pt,auto,node distance=1cm,
  thick,main node/.style={circle,draw,font=\bf}]

  \node[main node] (1) {};
  \node[main node] (2) [left of=1] {};
 \node[main node] (3) [below   of=2] {};
  \node[main node] (4) [ right of=3] {};
\node[main node] (5) [ right of=4] {};
\node[main node] (6) [below  of=4] {};
\node[main node] (7) [below  of=3] {};

  \path[every node/.style={font=\sffamily\small}]
    (1)  edge (5) edge(2)
    (2) edge  (1) edge (3) edge (4)
     (3)  edge (4)
    (4) edge  (1) edge (2) edge (3) edge(6) 
(5) edge(1) edge(4)
(6) edge (4)
(7) edge(3)
;
\end{tikzpicture}
\hskip 2cm
\begin{tikzpicture}[->,>=stealth',shorten >=1pt,auto,node distance=1cm,
  thick,main node/.style={circle,draw,font=\bf}]

  \node[main node] (1) {};
  \node[main node] (2) [left of=1] {};
 \node[main node] (3) [below   of=2] {};
  \node[main node] (4) [ right of=3] {};
\node[main node, fill=blue!20] (5) [ right of=4] {};
\node[main node] (6) [below  of=4] {};
\node[main node] (7) [below  of=3] {};

  \path[every node/.style={font=\sffamily\small}]
    (1)  edge (5) 
    (2) edge  (1) 
     (3)  edge (4)
    (4) edge (2) 
(6) edge (4)
(7) edge(3)

;
\end{tikzpicture}
\caption{An oriented graph, and a covering rooted tree}
\end{figure}
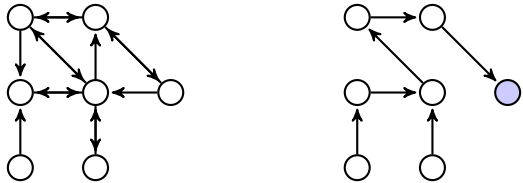

\subsection{Kirchhoff's matrix tree theorem}
We assume that the Markov chain is irreducible.
For  $i\in X$ let $Q^{(i)}$ be  the matrix obtained from $Q$ by deleting row and column $i$ and let $\mu(i)=\det(-Q^{(i)})$, then it is well known, and easy to see that $\mu$ is an invariant measure for $Q$. Indeed, if $Q^{(ij)}$ is obtained  by deleting row $i$ and column $j$, then  $\det(-Q^{(ij)})=\det(-Q^{(ii)})=\det(-Q^{(i)})$, since the sum of each line is 0, and
$\det(-Q)=\sum_i q_{ij}\det(-Q^{(ij)})=0$ for all $j$, by expanding the determinant along columns. That $\mu$ has positive entries follows from irreducibility and Kirchhoff's formula:
\begin{equation}\label{kirch}
\mu(i)=\sum_{t\in T_i} \pi(t)
\end{equation}
where the sum is over the set $T_i$ of oriented covering trees of $X$, rooted at $i$, and 
 $\pi(t)$ is the product of the $q_{kl}$ over all oriented edges  $(k,l)$ of the tree $t$.
See \cite{LP}, \S 4.
More generally, if $\{i_1,\ldots,i_k\}\subset X$, then Kirchhoff's formula also applies to the determinant of the matrix obtained from $Q$ by deleting columns and rows indexed by 
$i_1,\ldots,i_k$. This determinant is equal, up to a sign, to the sum over oriented covering forests, rooted at $i_1,\ldots,i_k$, of the product over edges of the forest.

\section{Lifting  the Markov chain to its covering trees}
\subsection{The lift}Notations are as in the preceding section furthermore
we assume that $Q$ is irreducible. The set of oriented covering rooted trees of $(X,E)$ is 
$T=\cup_{i\in X}T_i$. Let the map $p:T\to X$  assign to each tree $t$ its root (i.e. $p$ maps $T_i$  to $i$).
There exists an irreducible Markov chain on $T$ whose image by  $p$ is a Markov chain on $X$ with transition rates $Q$, and the vector  $(\pi(t))_{t\in T}$ is an invariant measure for this Markov chain. In particular by (\ref{invmeas}) the invariant measure $\pi$ projects by $p$ to the invariant measure $\mu$ and this construction provides a probabilistic interpretation of Kirchhoff's formula (\ref{kirch}).
 This Markov chain can be described by its transition rates $r_{st}, s,t\in T$.
 Let $s$ be a covering tree of $X$, rooted at $i$, and let $j\in X$ be such that 
 $q_{ij}>0$. There is a unique edge of $s$ coming out of $j$. Take out this edge from $s$  and then add the edge $(i,j)$.  One obtains a new oriented tree $t$, rooted at $j$ (see Figure 2 for an example). One puts then $r_{st}=q_{ij}$. For all pairs $s\ne t$ which are not obtained by this construction, one puts $r_{st}=0$. This defines a  unique matrix of transition rates 
$(r_{st})_{s,t\in T}$.

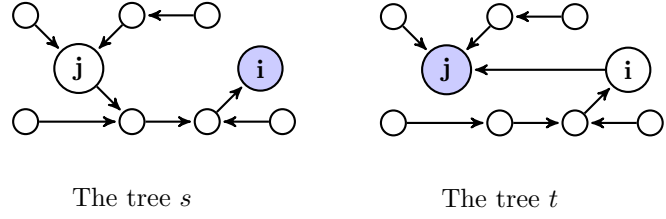
\begin{figure}
\centering
\begin{tikzpicture}[->,>=stealth',shorten >=1pt,auto,node distance=1cm,
  thick,main node/.style={circle,draw,font=\bf}]

  \node[main node] (1) {};
  \node[main node] (2) [below right of=1] {j};
 \node[main node] (3) [below  left of=2] {};
  \node[main node] (4) [below right  of=2] {};
\node[main node, ] (5) [ right of=4] {};
\node[main node] (6) [  right of=5] {};
\node[main node,fill=blue!20] (7) [above right  of=5] {i};
\node[main node] (8) [above right  of=2] {};
\node[main node] (9) [right  of=8] {};
\node[circle](10)[below of =4]{The tree $s$};

  \path[every node/.style={font=\sffamily\small}]
    (1)  edge (2) 
    (2) edge  (4) 
     (3)  edge (4)
    (4) edge (5) 
(5) edge (7)
(6) edge(5)
(8) edge (2)
(9) edge (8)
;
\end{tikzpicture}
\hskip 1cm
\begin{tikzpicture}[->,>=stealth',shorten >=1pt,auto,node distance=1cm,
  thick,main node/.style={circle,draw,font=\bf}]
 \node[main node] (1) {};
  \node[main node,fill=blue!20] (2) [below right of=1] {j};
 \node[main node] (3) [below  left of=2] {};
  \node[main node] (4) [below right  of=2] {};
\node[main node ] (5) [ right of=4] {};
\node[main node] (6) [  right of=5] {};
\node[main node] (7) [above right  of=5] {i};
\node[main node] (8) [above right  of=2] {};
\node[main node] (9) [right  of=8] {};
\node[circle](10)[below of =4]{The tree $t$};

  \path[every node/.style={font=\sffamily\small}]
    (1)  edge (2) 
    (7) edge  (2) 
     (3)  edge (4)
    (4) edge (5) 
(5) edge (7)
(6) edge(5)
(8) edge (2)
(9) edge (8)
;
\end{tikzpicture}
\caption{Lifting a transition between $i$ and $j$}
\end{figure}

It is clear that these transitions define a Markov chain which projects onto $M$ by the map $p$.

\begin{theorem}  The Markov chain with transition rates $R$ is irreducible, and the vector $\pi$ is an invariant measure for this Markov chain.
\end{theorem}
The proof can be found in \cite{AT}.

\subsection{An example}\label{ex}
Let $X=\{1,2,3\}$ and 
$$Q=\begin{pmatrix}\lambda&a&w\\u&\mu&b\\c&v&\nu\end{pmatrix}$$
with $\lambda=-a-w$, $\mu=-b-u$, $\nu=-c-v$. We assume that $a,b,c,u,v,w>0$.
  The graph $(X,E)$ looks as follows:

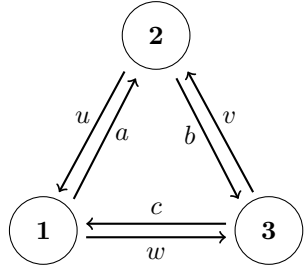
\begin{center}
\begin{tikzpicture}[scale=.15]
 \draw   (0,0) circle(3);\draw  (0,0) node{${\bf 1}$};
\draw (20,0) circle(3);\draw  (20,0) node{${\bf 3}$};
\draw (10,17.3) circle(3);\draw  (10,17.3) node{$ {\bf 2}$};

\draw[->,thick] (3.8,-0.6) -- (16.2,-0.6);\draw  (10,-2) node{$w$};
\draw[<-,thick] (3.8,0.6) -- (16.2,0.6);\draw  (10,2) node{$c$};

\draw[->,thick] (18.6,3.4) -- (12.8,14.1);\draw  (16.5,10) node{$v$};
\draw[<-,thick] (17.3,2.8) -- (11.8,13.5);\draw  (13,8.5) node{$b$};

\draw[<-,thick] (1.4,3.4) -- (7.2,14.1);\draw  (3.5,10) node{$u$};
\draw[->,thick] (2.7,2.8) -- (8.2,13.5);\draw  (7,8.5) node{$a$};
\end{tikzpicture}

\end{center}

Each covering rooted tree $t$ can be indexed by the monomial $\pi(t)$. There are   nine such covering trees: first
$cu,uv,bc$ rooted at $1$, then $av,ac,vw$ rooted at $2$, and finally $uw,bw, ab $ rooted at $3$.
With this ordering of $T$, the transition matrix for the lifted Markov chain is
$$R=\left(\begin{array}{ccccccccccc}
\lambda&0&0&\ &0&a&0&\ &w&0&0\\
0&\lambda&0&\ &a&0&0&\ &w&0&0\\
0&0&\lambda&\ &0&a&0&\ &0&w&0\\
\
0&u&0&\ &\mu&0&0&\ &0&0&b\\
u&0&0&\ &0&\mu&0&\ &0&0&b\\
0&u&0&\ &0&0&\mu&\ &0&b&0\\
\\
c&0&0&\ &0&0&v&\ &\nu&0&0\\
0&0&c&\ &0&0&v&\ &0&\nu&0\\
0&0&c&\ &v&0&0&\ &0&0&\nu
\end{array}\right)$$
Figure 3 shows the oriented graph. We have shown, for each vertex, its projection onto $X$ (namely $\bf 1$, $\bf 2$, or $\bf 3$) and for each oriented edge, its weight ($a,b,c,u,v$ or $w$).

\begin{figure}[!h]
\begin{tikzpicture}[scale=.55]
    \draw  (0,0) circle(.8);\draw  (0,0) node{$uv\ {\bf 1}$};
\draw(16,0) circle(.8);\draw  (16,0) node{$uw\ {\bf 3}$};
\draw(8,13.9) circle(.8);\draw  (8,13.9) node{$vw\ {\bf 2}$};

\draw  (6,3) circle(.8);\draw  (6,3) node{$ab\ {\bf 3}$};
\draw(10,3) circle(.8);\draw  (10,3) node{$ac\ {\bf 2}$};
\draw(8,6.5) circle(.8);\draw  (8,6.5) node{$bc\ {\bf 1}$};

\draw(8,10.2) circle(.8);\draw  (8,10.2) node{$bw\ {\bf 3}$};
\draw  (3,1.5) circle(.8);\draw  (3,1.5) node{$av\ {\bf 2}$};
\draw(13,1.5) circle(.8);\draw  (13,1.5) node{$cu\ {\bf 1}$};

\draw[->,thick] (1.3,-0.4) -- (14.7,-0.4);\draw  (7.5,0.2) node{$w$};
\draw[->,thick] (15.5,1.3) -- (8.7,13);\draw  (12.5,8) node{$v$};
\draw[->,thick] (7.3,13) -- (0.5,1.3);\draw  (3.5,8) node{$u$};

\draw[<-,thick] (7.1,3) -- (8.9,3) ;\draw  (8,3.5) node{$b$};
\draw[->,thick]  (8.6,5.5) -- (9.7,3.9);\draw  (9.7,4.8) node{$a$};
\draw[<-,thick] (7.4,5.5) -- (6.3,3.9) ;\draw  (6.3,4.8) node{$c$};

\draw[->,thick] (0.7,0.6) -- (2.1,1.3) ;\draw  (1.1,1.2) node{$a$};
\draw[<-,thick] (0.8,0.3) -- (2.2,1) ;\draw  (1.8,0.5) node{$u$};
\draw[->,thick] (3.7,2.1) -- (5.1,2.8) ;\draw  (4.1,2.7) node{$b$};
\draw[<-,thick] (3.8,1.8)  --(5.2,2.5) ;\draw  (4.8,2) node{$v$};

\draw[->,thick] (15.3,0.6) -- (13.9,1.3) ;\draw  (14.9,1.2) node{$c$};
\draw[<-,thick] (15.2,0.3)  --(13.8,1) ;\draw  (14.2,0.5) node{$w$};

\draw[->,thick] (12.3,2.1) -- (10.9,2.8) ;\draw (11.9,2.7) node{$a$};
\draw[<-,thick] (12.2,1.8) -- (10.8,2.5) ;\draw  (11.2,2) node{$u$};

\draw[<-,thick] (7.85,7.5) -- (7.85,9.2) ;\draw  (8.5,8.3) node{$w$};
\draw[->,thick] (8.15,7.5) -- (8.15,9.2) ;\draw  (7.35,8.3) node{$c$};
\draw[<-,thick] (7.85,11.2)  --(7.85,12.9) ;\draw  (8.5,12) node{$v$};
\draw[->,thick] (8.15,11.2)  --(8.15,12.9) ;\draw  (7.35,12) node{$b$};

\end{tikzpicture}
\caption{The graph $T$ }
\end{figure}
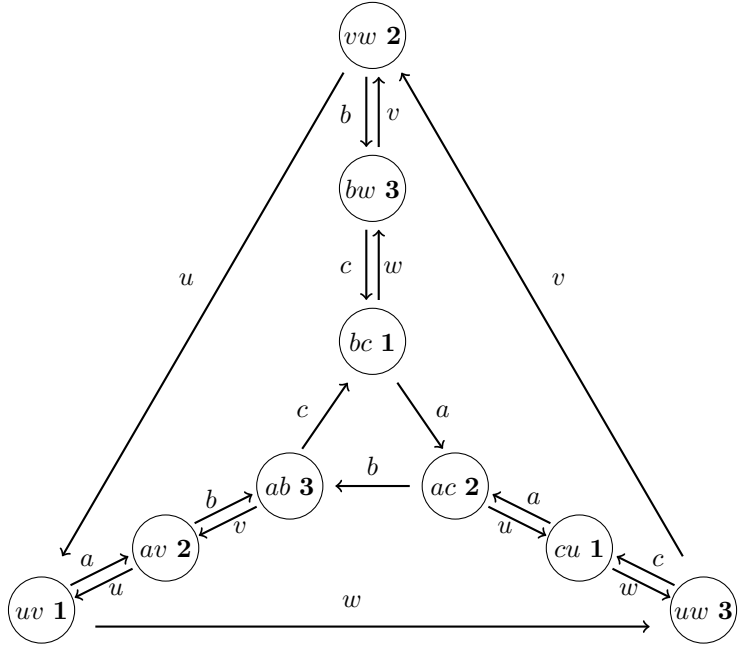

\medskip

\section{A polynomial associated to the Markov chain} 
\subsection{The polynomial}
We consider, as in the previous sections, an irreducible Markov chain on a finite set $X$ with transition matrix $Q$ and its canoncial lift to $T$, with transition matrix $R$.
For $t\in T$, consider the  matrix $R^{(t)}$ obtained from $R$
 by taking out row and column $t$, and let $\rho(t)=\det(-R^{(t)})$, then $\rho$ is an invariant measure for $R$, and gives a generating function for covering trees of the graph $T$.
 If we fix the graph $(X,E)$, then $\rho(t)$ is a polynomial in the variables $q_{ij}$, where we keep only the pairs $(i,j)$ forming an edge in $E$. 
Since $\pi$ and $\rho$ are  invariant measures of the lifted Markov chain, they are proportional so that there exists 
there exists a function,  $\Psi(q_{ij})$, independent of $t$, such that 
 for all $t\in T$, 
$$\rho(t)=\pi(t)\Psi$$
Actually it is not difficult to see that $\Psi(q_{ij})$ is a polynomial.
Indeed one has $\Psi=\rho(t)/\pi(t)$, and $\pi(t)$ is a monomial so that, by reducing, 
 $\Psi=P/m$ with $P$ a polynomial and $m$ a monomial prime with $P$. In particular, 
$\rho(t)=\pi(t)P/m$ is a polynomial for all $t$, hence $m$ divides $\pi(t)$ for all $t$. But the $\pi(t)$ have no common divisor, since a variable $q_{kl}$ cannot divide $\pi(t)$ is $t$ is rooted at $k$, therefore $m=1$.

\subsection{Some examples}
If $|X|=3$, with the notations of section \ref{ex}, one can compute 
\begin{eqnarray*}
\Psi(a,b,c,u,v,w)&
=&(bc+cu+uv)(av+ac+vw)(ab+bw+uw)\\
&=&\prod_{i\in X}\left(\sum_{t\in T_i}\pi(t)\right)
\end{eqnarray*}
so that $\Psi$ is the product of all symmetric rank two minors of the matrix $-Q$
(a {\sl  symmetric minor of rank $k$} of a matrix of size $n$, is   the determinant of a submatrix obtained by deleting $n-k$ rows and the $n-k$ columns with the same indices).
I have computed the polynomial $\Psi$ for various graphs with $4$ vertices and found in many cases that $\Psi$ can be written as a product of symmetric minors of the matrix $-Q$. I could not compute in the case of
$|X|=4$ and the graph $(X,E)$ is a complete graph, but by putting some of the variables equal to $1$ to make the determinant easier to compute,  the results suggest that the formula for $\Psi$ in this case should be
$$\Psi=m_2(Q)^3m_3(Q)^2$$
where  $m_k(Q)$  is the product of all symmetric  minors of rank $k$ of $-Q$.

Based on this small evidence it seems natural to conjecture that for any irreducible graph $(X,E)$ the polynomial $\Psi$ should be a product of symmetric minors of the matrix $-Q$. Which minors appear, and what are their exponents, should depend on the graph and encode some of its geometry.
By symmetry, 
in the case of a complete graph on $n$ vertices, the result should be a product $\prod_{k=1}^{n-1}m_k^{v_k^n}$ for some exponents $v_k^n$.
Guillaume Chapuy \cite{C} has done some further computations for $n=5$ and conjectured that 
$v_k^n=(k-1)(n-1)^{n-k-1}$.
One can check that, at least, this gives the correct degree. In general the degree of $\Psi$ is $|T|-n$, and in the case of a complete graph, $|T|=n^{n-1}$, moreover there are $\binom{n}{k}$ symmetric minors of rank $k$, which are polynomials of degree $k$, and 
$$\sum_{k=2}^{n-1}\binom{n}{k}k(k-1)(n-1)^{n-k-1}=n^{n-1}-n$$
as follows easily from the binomial formula.

In the following I obtain a result for the case where the graph is a ring: $X=\{1,2,\ldots,n\}$ and the edges are $(i,i\pm 1)$ (where $i\pm 1$ is taken modulo $n$).

\begin{theorem}\label{th} If $(X,E)$ is a ring of size $n\geq 3$, then $\Psi$ is the product of the symmetric minors of size $n-1$:
$$\Psi=m_{n-1}(Q)$$
\end{theorem}
The proof of Theorem \ref{th}, which is the main result of this paper, occupies the next section.
\section{Proof of  Theorem \ref{th}}
In this section, $(X,E)$ denotes a ring, namely, $X=\{1,2,\ldots,n\}$ and the edges are $(i,i\pm 1)$ (here and in the sequel $i\pm 1$ is always taken modulo $n$). I will illustrate this with $n=4$, as in Figure 4.

\medskip

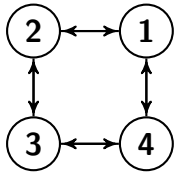
\begin{figure}
\begin{tikzpicture}[->,>=stealth',shorten >=1pt,auto,node distance=1.5cm,
  thick,main node/.style={circle,draw,font=\sffamily\Large\bfseries}]

  \node[main node] (1) {1};
  \node[main node] (2) [left of=1] {2};
 \node[main node] (3) [below   of=2] {3};
  \node[main node] (4) [below of=1] {4};

  \path[every node/.style={font=\sffamily\small}]
    (1) edge    (2) edge (4)
    (2) edge (1) edge (3)
    (3) edge  (2) edge (4)
    (4)         edge  (1) edge (3);
\end{tikzpicture}
\caption{The ring $(X,E)$ with $n=4$}
\end{figure}

\medskip
\subsection{Structure of the graph $T$}
For each  pair $(i,j)\in X^2$ there exists a unique covering tree of $(X,E)$, rooted at $i$, which has no edge between $j$ and $j+1$. Let us  denote this covering rooted tree by $[i,j]$.
For example, if $n=4$ here are the trees denoted by, respectively, $[2,3]$ and $[3,3]$ (here and in the sequel the roots are shaded):

\medskip

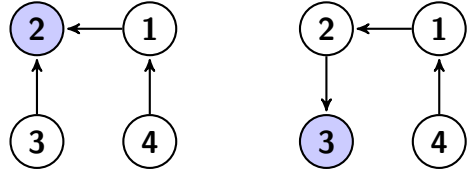
\begin{center}
\begin{tikzpicture}[->,>=stealth',shorten >=1pt,auto,node distance=1.5cm,
  thick,main node/.style={circle,draw,font=\sffamily\Large\bfseries}]

  \node[main node] (1) {1};
  \node[main node, fill=blue!20] (2) [left of=1] {2};
 \node[main node] (3) [below   of=2] {3};
  \node[main node] (4) [below of=1] {4};

  \path[every node/.style={font=\sffamily\small}]
    (1) edge    (2)
    (3) edge  (2)
    (4)         edge  (1);
\end{tikzpicture}
\hskip 1.5cm
\begin{tikzpicture}[->,>=stealth',shorten >=1pt,auto,node distance=1.5cm,
  thick,main node/.style={circle,draw,font=\sffamily\Large\bfseries}]

  \node[main node] (1) {1};
  \node[main node] (2) [left of=1] {2};
 \node[main node, fill=blue!20] (3) [below   of=2] {3};
  \node[main node] (4) [below of=1] {4};

  \path[every node/.style={font=\sffamily\small}]
    (1) edge    (2)
    (2) edge  (3)
    (4)         edge  (1);
\end{tikzpicture}
\end{center}

\medskip
It is easy to check that these are all covering rooted trees of $(X,E)$, in 
particular  $|T|=n^2$. Let us now describe the structure of the graph on $T$ induced by the lifting of the Markov chain.

First consider 
the trees indexed by the pairs $[i,i]$. The trees $[i,i]$ and $[i+1,i+1]$ are connected by an edge labelled $q_{i,i+1}$ e.g.

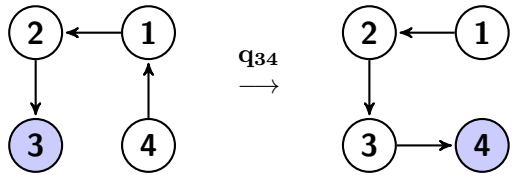
\medskip

\begin{center}
\begin{tikzpicture}[->,>=stealth',shorten >=1pt,auto,node distance=1.5cm,
  thick,main node/.style={circle,draw,font=\sffamily\Large\bfseries}]

  \node[main node] (1) {1};
  \node[main node] (2) [left of=1] {2};
 \node[main node, fill=blue!20] (3) [below   of=2] {3};
  \node[main node] (4) [below of=1] {4};

  \path[every node/.style={font=\sffamily\small}]
    (1) edge    (2)
    (2) edge  (3)
    (4)         edge  (1);
\end{tikzpicture}
$\qquad\begin{matrix} \bf q_{34}\\ \longrightarrow\\ \\ \\ \\ \\ \\ \\ \end{matrix}\qquad$
\begin{tikzpicture}[->,>=stealth',shorten >=1pt,auto,node distance=1.5cm,
  thick,main node/.style={circle,draw,font=\sffamily\Large\bfseries}]

  \node[main node] (1) {1};
  \node[main node] (2) [left of=1] {2};
 \node[main node] (3) [below   of=2] {3};
  \node[main node, fill=blue!20] (4) [below of=1] {4};

  \path[every node/.style={font=\sffamily\small}]
    (1) edge    (2)
    (2) edge  (3)
    (3)         edge  (4);
\end{tikzpicture}
\end{center}

These trees  form an oriented ring in $T$:

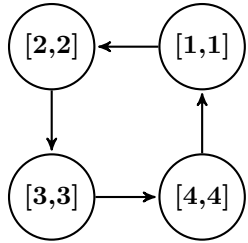
\medskip

\begin{center}
\begin{tikzpicture}[->,>=stealth',shorten >=1pt,auto,node distance=2cm,
  thick,main node/.style={circle,draw,font=\bf}]

  \node[main node] (1) {[1,1]};
  \node[main node] (2) [ left of=1] {[2,2]};
 \node[main node] (3) [below   of=2] {[3,3]};
  \node[main node] (4) [below  of=1] {[4,4]};

  \path[every node/.style={font=\sffamily\small}]
    (1) edge    (2)
    (2) edge  (3)
     (3) edge  (4)
    (4) edge  (1);
\end{tikzpicture}
\end{center}

\medskip

The trees indexed by pairs $[i,i-1]$ are connected by edges labelled $q_{i,i-1}$:

\medskip

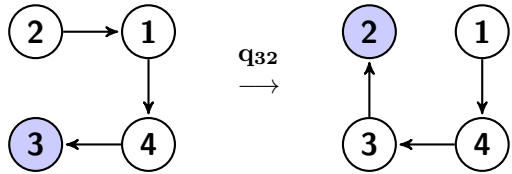
\begin{center}
\begin{tikzpicture}[->,>=stealth',shorten >=1pt,auto,node distance=1.5cm,
  thick,main node/.style={circle,draw,font=\sffamily\Large\bfseries}]

  \node[main node] (1) {1};
  \node[main node] (2) [left of=1] {2};
 \node[main node, fill=blue!20] (3) [below   of=2] {3};
  \node[main node] (4) [below of=1] {4};

  \path[every node/.style={font=\sffamily\small}]
    (1) edge    (4)
    (2) edge  (1)
    (4)         edge  (3);
\end{tikzpicture}
$\qquad\begin{matrix}\bf q_{32}\\ \longrightarrow\\ \\ \\ \\ \\ \\ \\\end{matrix}\qquad$
\begin{tikzpicture}[->,>=stealth',shorten >=1pt,auto,node distance=1.5cm,
  thick,main node/.style={circle,draw,font=\sffamily\Large\bfseries}]

  \node[main node] (1) {1};
  \node[main node, fill=blue!20] (2) [left of=1] {2};
 \node[main node] (3) [below   of=2] {3};
  \node[main node] (4) [below of=1] {4};

  \path[every node/.style={font=\sffamily\small}]
    (1) edge    (4)
    (4) edge  (3)
    (3)         edge  (2);
\end{tikzpicture}
\end{center}

They form another oriented ring:

\medskip
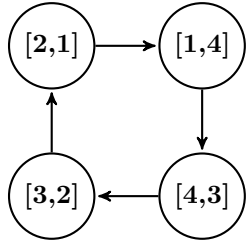
\begin{center}
\begin{tikzpicture}[->,>=stealth',shorten >=1pt,auto,node distance=2cm,
  thick,main node/.style={circle,draw,font=\bf}]

  \node[main node] (1) {[1,4]};
  \node[main node] (2) [ left of=1] {[2,1]};
 \node[main node] (3) [below   of=2] {[3,2]};
  \node[main node] (4) [below  of=1] {[4,3]};

  \path[every node/.style={font=\sffamily\small}]
    (1) edge    (4)
    (2) edge  (1)
     (3) edge  (2)
    (4) edge  (3);
\end{tikzpicture}
\end{center}

\medskip

There are also edges in the two directions between $[i,j]$ and $[i+1,j]$, labelled by $q_{i,i+1}$ and $q_{i+1,i}$:

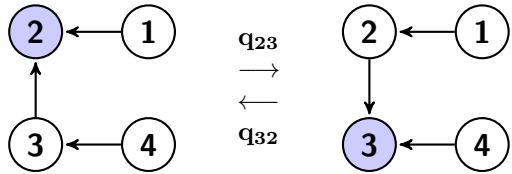
\medskip

\begin{center}
\begin{tikzpicture}[->,>=stealth',shorten >=1pt,auto,node distance=1.5cm,
  thick,main node/.style={circle,draw,font=\sffamily\Large\bfseries}]

  \node[main node] (1) {1};
  \node[main node,fill=blue!20] (2) [left of=1] {2};
 \node[main node, ] (3) [below   of=2] {3};
  \node[main node] (4) [below of=1] {4};

  \path[every node/.style={font=\sffamily\small}]
    (1) edge    (2)
    (4) edge  (3)
    (3)         edge  (2);
\end{tikzpicture}
$\qquad\begin{matrix} \bf  q_{23}\\ \longrightarrow\\ \longleftarrow\\\bf q_{32} \\ \\ \\ \\ \\ \\ \end{matrix}\qquad$
\begin{tikzpicture}[->,>=stealth',shorten >=1pt,auto,node distance=1.5cm,
  thick,main node/.style={circle,draw,font=\sffamily\Large\bfseries}]

  \node[main node] (1) {1};
  \node[main node] (2) [left of=1] {2};
 \node[main node, fill=blue!20] (3) [below   of=2] {3};
  \node[main node] (4) [below of=1] {4};

  \path[every node/.style={font=\sffamily\small}]
    (1) edge    (2)
    (2) edge  (3)
    (4)         edge  (3);
\end{tikzpicture}
\end{center}

These form lines of length $n$:

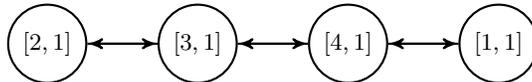
\begin{center}
\begin{tikzpicture}[->,>=stealth',shorten >=1pt,auto,node distance=2cm,
  thick,main node/.style={circle,draw,font=\sffamily\small\bfseries}]

  \node[main node] (1) {$[2,1]$};
  \node[main node] (2) [right of=1] {$[3,1]$};
\node[main node] (3) [right of=2]{$[4,1]$};
  \node[main node] (5) [ right of=3] {$[1,1]$};

  \path[every node/.style={font=\sffamily\small}]
    (1) edge   (2)
        
    (2) edge  (1)
        edge  (3)
    (3) edge  (2)
        edge  (5)
    
     (5)edge (3)
      ;

\end{tikzpicture}
\end{center}

One can represent the graph $T$ by putting two concentric oriented rings  of size $n$, with  opposite orientations,
and joining the vertices of the rings by sequences of vertices connected by double edges, see Figure 3 for $n=3$ and Figure 5 for $n=4$:

\begin{figure}
\begin{tikzpicture}[->,>=stealth',shorten >=1pt,auto,node distance=1cm,
  thick,main node/.style={circle,draw,font=\bf}]

  \node[main node] (1) {};
  \node[main node] (2) [below right of=1] {};
 \node[main node] (3) [below  right of=2] {};
  \node[main node] (4) [below right of=3] {};
\node[main node] (5) [ right of=4] {};
\node[main node] (6) [above right of=5] {};
\node[main node] (7) [above right of=6] {};
\node[main node] (8) [above right of=7] {};
\node[main node] (9) [below of=4] {};
\node[main node] (10) [below left of=9] {};
\node[main node] (11) [below left of=10] {};
\node[main node] (12) [below left of=11] {};
\node[main node] (13) [below of=5] {};
\node[main node] (14) [below right of=13] {};
\node[main node] (15) [below right of=14] {};
\node[main node] (16) [below right of=15] {};

  \path[every node/.style={font=\sffamily\small}]
    (1) edge    (2)
    (1)edge    (8) 
    (2) edge  (1)
    (2) edge  (3)
     (3) edge  (2)
     (3) edge  (4)
    (4) edge  (3)
     (4) edge  (9)
   (5) edge  (4)
    (5) edge  (6)
    (6) edge  (5)
    (6) edge  (7)
    (7) edge  (6)
    (7) edge  (8)
     (8) edge  (7)
    (8) edge  (16)
    (9) edge  (10)
    (9) edge  (13)
    (10) edge  (9)
    (10) edge  (11)
    (11) edge  (10)
     (11) edge  (12)
     (12) edge  (11)
     (12) edge  (1)
      (13) edge  (14)
       (13) edge  (5)
      (14) edge  (13)
      (14) edge  (15)
        (15) edge  (16)
        (15) edge  (14)
       (16) edge  (15)
       (16) edge  (12);

\end{tikzpicture}
\caption{The graph of $T$ for $n=4$}
\end{figure}
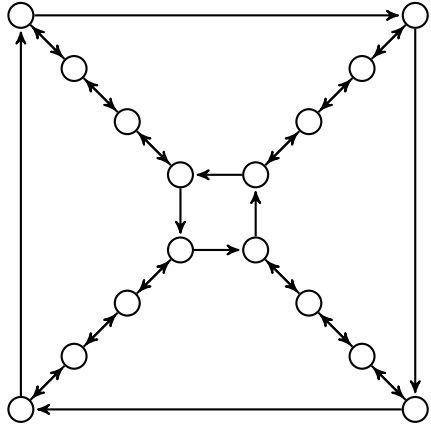

\subsection{The symmetric $n-1$ minors of $-Q$}
We will use the following lemma.

\begin{lemma}The symmetric $n-1$ minors of $-Q$ are prime polynomials.
\end{lemma}
\begin{proof}
Let $i\in Q$, then $\det(-Q^{(i)})$, the symmetric $n-1$ minor corresponding to $i$ is a polynomial with degree at most one in each variable. More precisely, using  Kirchhoff's formula this minor is the generating function of covering trees rooted at $i$ and  it can be written as 
$\alpha q_{i-1,i}+\beta$ where $\alpha$ and $\beta$ are  polynomials of degree $0$ in $q_{i-1,i}$. Moreover   $\beta$ is a monomial since there exists a unique covering tree of $(X,E)$ rooted at $i$ which does not contain the edge $(i-1,i)$. It follows that any nontrivial factorisation of this polynomial can be written as

\begin{equation}\label{factor}
\alpha q_{i-1,i}+\beta=(\gamma q_{i-1,i}+\delta)\eta
\end{equation}

 where $\gamma,\delta,\eta$ have degree $0$ in $q_{i-1,i}$ and  $\eta\delta=\beta$. In particular, $\eta$ is a nontrivial monomial, therefore there exists a variable $q_{kl}$ which divides $\alpha q_{i-1,i}+\beta$, and this means that the edge $(k,l)$ belongs to all covering trees rooted at $i$. Clearly this is not possible, therefore a nontrivial factorisation such as (\ref{factor}) does not exist, and the symmetric minor is a prime polynomial.

\end{proof}

\subsection{A preliminary lemma}
Consider the restriction of the graph $T$ to the sets of vertices $$G=\{[1,n],[2,1],[3,1],\ldots,[n,1],[1,1]\}
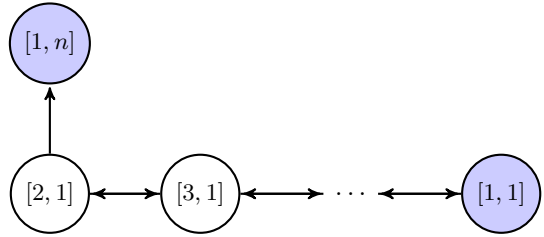
$$ 

\medskip

\begin{center}
\begin{tikzpicture}[->,>=stealth',shorten >=1pt,auto,node distance=2cm,
  thick,main node/.style={circle,draw,font=\sffamily\small\bfseries}]

  \node[main node] (1) {$[2,1]$};
  \node[main node] (2) [right of=1] {$[3,1]$};
\node[circle] (3) [right of=2]{$\ldots$};
 %\node[main node] (4) [right  of=3] {$[n,1]$};
  \node[main node,fill=blue!20] (5) [ right of=3] {$[1,1]$};
  \node[main node,fill=blue!20] (6) [above of=1]{$[1,n]$};

  \path[every node/.style={font=\sffamily\small}]
    (1) edge   (2)
        edge  (6)
    (2) edge  (1)
        edge  (3)
    (3) edge  (2)
        edge  (5)
    
     (5)edge (3)
      ;

\end{tikzpicture}
\end{center}

\medskip

and $$H=
\{[1,n],[2,n],[3,n],\ldots,[n,n],[1,1]\}
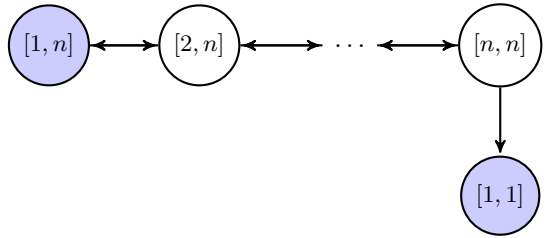
$$

\medskip

\begin{center}
\begin{tikzpicture}[->,>=stealth',shorten >=1pt,auto,node distance=2cm,
  thick,main node/.style={circle,draw,font=\sffamily\small\bfseries}]

  \node[main node,fill=blue!20] (6){$[1,n]$};
\node[main node] (7) [right of=6]{$[2,n]$};
\node[circle] (8) [right of=7]{$\ldots$};
%\node[main node] (9) [right of=8]{$[n-1,n]$};
\node[main node] (10) [right of=8]{$[n,n]$};

\node[main node,fill=blue!20] (11) [below of=10]{$[1,1]$};

  \path[every node/.style={font=\sffamily\small}]
    
      (6) edge (7)
          
     (7) edge  (8)
  edge (6)
(8) edge (7)
    edge (10)
(10)  edge (8)
  edge(11)

;

\end{tikzpicture}
\end{center}
We will need the following lemma.
\begin{lemma}\label{covfor}

The generating function of the set of covering forests  
of $G$, rooted at $[1,n]$ and $[1,1]$ is equal to $\det(-Q^{(1)})$, the generating function for the set of covering  trees of $X$, rooted at 1.
The same is true with $H$ instead of $G$.

\end{lemma}
\begin{proof}
One can check easily that the restriction of the projection $p$ to $G$ induces a bijection between the  covering forests  of $G$ rooted at $[1,n]$ and $[1,1]$ and the covering trees of $X$ rooted at $1$ (observe that $[1,n]$ and $[1,1]$ both project to $1$), and this bijection preserves the labels of the edges. The same is true for $H$ and the lemma follows.
\end{proof}

\subsection{}
We will now prove that the symmetric minor $\det(-Q^{(i)})$ divides the symmetric minor
$\det(-R^{([i,i])})$. By symmetry it is enough to prove this for $i=1$. By Kirchhoff's formula, we know that the polynomial $\det(-R^{([1,1])})$
is the generating polynomial of the covering trees  of $T$ rooted at vertex $[1,1]$.

Let  $K=G\cup H$ and let $L=T\setminus K$.
The part of the graph $T$ containing   $K$ looks like

\medskip

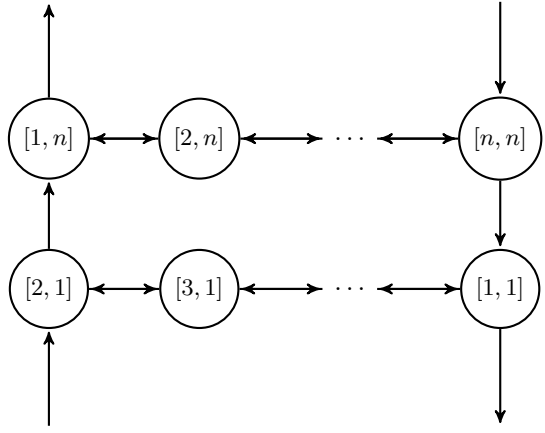
\begin{center}
\begin{tikzpicture}[->,>=stealth',shorten >=1pt,auto,node distance=2cm,
  thick,main node/.style={circle,draw,font=\sffamily\small\bfseries}]

  \node[main node] (1) {$[2,1]$};
  \node[main node] (2) [right of=1] {$[3,1]$};
\node[circle] (3) [right of=2]{$\ldots$};
 %\node[main node] (4) [right  of=3] {$[n,1]$};
  \node[main node] (5) [ right of=3] {$[1,1]$};
  \node[main node] (6) [above of=1]{$[1,n]$};
\node[main node] (7) [right of=6]{$[2,n]$};
\node[circle] (8) [right of=7]{$\ldots$};
%\node[main node] (9) [right of=8]{$[n-1,n]$};
\node[main node] (10) [right of=8]{$[n,n]$};
\node[circle] (11) [above of=6]{};
\node[circle] (12) [above of=10]{};
\node[circle] (13) [below of=1]{};
\node[circle] (14) [below of=5]{};
  \path[every node/.style={font=\sffamily\small}]
    (1) edge   (2)
        edge  (6)
    (2) edge  (1)
        edge  (3)
    (3) edge  (2)
        edge  (5)
    
     (5)edge (3)
        edge (14)
      (6) edge (7)
          edge (11)
     (7) edge  (8)
  edge (6)
(8) edge (7)
    edge (10)
(10)  edge (8)
      edge (5)
(12) edge (10)
(13) edge (1)

;

\end{tikzpicture}
\end{center}

\medskip

 Observe that 
the only way one can enter the set $K$ by a path coming from $L$ is through the vertices $[2,1]$ or $[n,n]$.
 Let now $\tau$ be a covering tree of 
$T$, rooted at $[1,1]$. If we consider the set of vertices $L\cup\{[2,1],[n,n]\}$ together with the edges of $\tau$ coming out of elements of  $L$, we obtain two disjoint trees, rooted respectively at $[n,n]$ and $[2,1]$. Let us now fix such a pair of  trees $A$ and $B$, and consider the set of covering trees $\tau$ of $T$, rooted at $[1,1]$, which induce the pair $(A,B)$. There are three possibilities for  the edge coming out of $[1,n]$ in such a tree:

\medskip

$i)$ it connects to $[2,n]$

$ii)$ it connects to $[n,n-1]$ which belongs to $A$

$iii)$ it connects to $[n,n-1]$ which belongs to $B$.

\medskip

If we are in the first case then the restriction of the tree to $G$ forms 
 a covering forest of $G$, rooted at $[1,n]$ and $[1,1]$. Furthermore  any such forest can occur, independently of the trees $A$ and $B$. It follows that the generating function of trees in case $i)$ is a multiple of the generating function of such covering forests, which is $\det(-Q^{(1)})$ by Lemma \ref{covfor}.

In case $ii)$ the same argument as in $i)$ can be applied, so we conclude again that the generating function of such trees is a multiple of $\det(-Q^{(1)})$.

Finally in case $iii)$ the edge $([2,1],[1,n])$ cannot belong to the tree, but  a similar reasoning, this time with $H$ instead of $G$, shows that  
the generating function of such trees is a again multiple of $\det(-Q^{(1)})$.

From this, summing over all three cases, and all pairs $(A,B)$ we conclude that $\det(-R^{([1,1])})$, the generating function of the set of covering trees of $T$, rooted at $[1,1]$, is a multiple of  $\det(-Q^{(1)})$. Since $\det(-R^{([1,1])})=\pi([1,1])\Psi$ and $\pi([1,1])$ is a monomial which is prime with $\det(-Q^{(1)})$ it follows that $\det(-Q^{(1)})$ divides the polynomial $\Psi$. By symmetry, this is true of all the $\det(-Q^{(i)})$, for $i\in X$  and since these are distinct prime polynomials, we conclude that $\Psi$ is a multiple of  $m_{n-1}=\prod_i\det(-Q^{(i)})$. The degree of the polynomial $\det(-R^{([1,1])})$ is $n^2-1$, the degree of   $m_{n-1}$ is $n(n-1)$ and the degree of $\pi([1,1])$ is $n-1$. It follows that $\Psi$ and $m_{n-1}$ are proportional.

 In order to find the constant of proportionality, we consider  the generating function of the covering trees of $T$, rooted at $[n,n]$. This generating function is 
$\det(-R^{([n,n])})=\pi([n,n])\Psi$. I claim that the 
coefficient of the monomial 
\begin{equation}\label{monom}
q_{n1}^{n-1}\prod_{i=1}^{n-1} q_{i,i+1}^{n}
\end{equation}
 in $\det(-R^{([n,n])})$ is 1. Indeed for each $i\leq n$ there are exactly $n$ edges in $T$ which are labelled $q_{i,i+1}$, and one of the edges labelled $q_{n1} $ goes out of $[n,n]$ so it cannot belong to a tree rooted at $[n,n]$, therefore there exists at most one covering tree rooted at $[n,n]$ whose product over labelled edges is equal to (\ref{monom}).  On the other hand, one can check that, taking the graph formed with all these edges, one obtains a covering tree rooted at $[n,n]$, see e.g. Figure 6 for the case of $n=4$.

\begin{figure}[!h]
\begin{tikzpicture}[->,>=stealth',shorten >=1pt,auto,node distance=1cm,
  thick,main node/.style={circle,draw,font=\bf}]

  \node[main node] (1) {};
  \node[main node] (2) [below right of=1] {};
 \node[main node] (3) [below  right of=2] {};
  \node[main node] (4) [below right of=3] {};
\node[main node] (5) [ right of=4] {};
\node[main node] (6) [above right of=5] {};
\node[main node] (7) [above right of=6] {};
\node[main node] (8) [above right of=7] {};
\node[main node] (9) [below of=4] {};
\node[main node] (10) [below left of=9] {};
\node[main node] (11) [below left of=10] {};
\node[main node] (12) [below left of=11] {};
\node[main node,fill=blue!20] (13) [below of=5] {};
\node[main node] (14) [below right of=13] {};
\node[main node] (15) [below right of=14] {};
\node[main node] (16) [below right of=15] {};

  \path[every node/.style={font=\sffamily\small}]
    (1) edge    (2)
    (2) edge  (3)
     (3) edge  (4)
     (4) edge  (9)
    (8) edge  (7)
     (7) edge  (6)
   (5) edge  (4)
    (12) edge  (11)
    (6) edge  (5)
    (11) edge  (10)
    (10) edge  (9)
    (9) edge  (13)
     (16) edge  (15)
    (15) edge  (14)
    (14) edge  (13);

\end{tikzpicture}
\caption{The covering tree for $n=4$}
\end{figure}
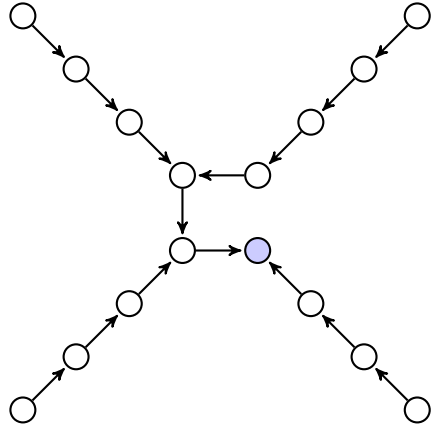

It remains now to check that the coefficient of $\pi([n,n])\prod_i\det(-Q^{(i)})$ is 1. This follows from the fact that for each $i$ there exists a unique covering tree of $X$ rooted at $i$, whose labels are all of the form $q_{k,k+1}$. Taking the product over these trees one recovers the product (\ref{monom}). 

This completes the proof of Theorem \ref{th}.\qed

\subsection{Final remark}
If we look at formula
$$\det(-R^{([n,n])})=\pi([n,n])\prod_{i=1}^n\det(-Q^{(i)})$$
there is a combinatorial significance for both sides of the equality. The left hand sides is the generating function for covering trees of $T$ rooted a $[n,n]$ whereas the right hand side is the  generating function of  the $n$-tuples of rooted covering trees of $(X,E)$  rooted at $1,2,\ldots,n$. It would be interesting to tranform our proof of this formula into a bijective proof by exhibiting a bijection between these two sets which respects the weights. This could shed some light on the general case.

\end{document}